\documentclass[10pt]{article}
\usepackage[latin9]{inputenc}
\usepackage[T1]{fontenc}
\usepackage[french,english]{babel}
\usepackage{graphicx}
\usepackage[top=3cm, bottom=3cm, left=3cm, right=3cm]{geometry}
\usepackage{float}
\usepackage{amssymb}
\usepackage{amsmath}
\usepackage{titlepic}
\usepackage{caption} 
\usepackage{epstopdf}
\usepackage{xfrac}
\usepackage{subfigure}
\usepackage{bm}
\usepackage{esint}
\usepackage{mathtools}
\usepackage[dvipsnames]{xcolor}

\usepackage{amsthm}

\newcommand{\R}{\mathbb{R}}
\newcommand{\N}{\mathbb{N}}

\newcommand{\nubf}{\boldsymbol{\nu}}

\newcommand{\pbf}{\mathbf{p}}

\newcommand{\Ebf}{\mathbf{E}}

\newcommand{\dd}{\mathrm{d}}
\newcommand{\tin}{\text{ in }}

\newtheorem{theorem}{Theorem}[section]
\newtheorem{remark}{Remark}[section]

\newtheorem{definition}{Definition}[section]
\newtheorem{lemma}{Lemma}[section]

\newtheorem{proposition}{Proposition}[section]

\begin{document}

\title{Super-localisation of a point-like emitter in a resonant environment : correction of the mirage effect}
\author{Lorenzo Baldassari\thanks{\footnotesize Department of Computational Applied Mathematics and Operations Research, 
Rice University, 
Houston, TX 77005 USA.}   \and Pierre Millien\thanks{\footnotesize  ESPCI Paris, PSL Research University, CNRS, Institut Langevin, France }\and Alice L. Vanel\thanks{\footnotesize Department of Mathematics, 
ETH Z\"urich, 
R\"amistrasse 101, CH-8092 Z\"urich, Switzerland.} }

\date{\today}

\maketitle

\begin{abstract}
In this paper,  we show that it is possible to overcome one of the fundamental limitations of super-resolution microscopy techniques: the necessity to be in an \emph{optically homogeneous} environment.
Using recent modal approximation results from \cite{baldassari2021modal,AMMARI2022676} we show as a proof of concept that it is possible to recover the position of a single point-like emitter in a \emph{known resonant environment} from far-field measurements with a precision two orders of magnitude below the classical Rayleigh limit. The procedure  does not involve solving any partial differential equation,  is computationally light  (optimisation in $\R^d$ with $d$ of the order of $10$) and therefore suited for the recovery of a very large number of single emitters. 
\end{abstract}

\def\keywords2{\vspace{.5em}{\textbf{  Mathematics Subject Classification
(MSC2000).}~\,\relax}}
\def\endkeywords2{\par}
\keywords2{35R30, 35C20.}

\def\keywords{\vspace{.5em}{\textbf{ Keywords.}~\,\relax}}
\def\endkeywords{\par}
\keywords{plasmonic resonance, biomedical imaging, mirage effect, super resolution,  super localisation}

\section{Introduction}
\subsection{Context and position of the problem}

Super-resolution imaging \cite{huang2009super,hell20152015} has had a huge impact on biological imaging (its contribution was recognised by the Nobel Prize in $2014$ \cite{ehrenberg2014nobel}).  With these techniques,  biological structures can be imaged with a recovery of details that are two order of magnitude below the classical diffraction limit of microscopy. The technique is now being used in other areas of material science such as polymer dynamics study \cite{gramlich2014fluorescence,urban2016subsurface} or plasmonics \cite{bouchet2019probing}. We refer the reader to the review paper \cite{Willets:2017aa} and references therein for an overview of the recent research on plasmonics using super-resolved microscopy. Super resolution microscopy's principle is to excite a single point-like emitter in a diffraction-limited spot, and to find the position of the emitter by fitting a point spread function over the spot. The procedure is iterated several thousand times to get a precise image, similar to the technique of \emph{pointillism} in painting. 
There are several types of microscopy depending on whether the excitation of the emitters is tailored on their emissive properties. The latter is what can be called localisation-based super-resolution imaging, and is known under several acronyms such as STORM, PALM, PAINT, \ldots

Single emitters such as quantum dots or fluorescent molecules interact with their environment in several ways.

It has been known that the lifetime emission of single fluorescent molecules depends on the local density of states (LDOS) and therefore they are excellent probes for the near-field around a nano-structure. In~\cite{bouchet2019probing} the authors experimentally measured the lifetime of single emitters around a plasmonic rod while simultaneously super-localising the emitters. They were able to construct a map of the fluorescence lifetime with a resolution approaching $20$ nm.

However, the single emitter can also excite the plasmonic resonances of the nano-structure nearby if its emission frequency is close to a plasmonic resonant frequency. In this case, the far-field image does not only consist of the point spread function of the single emitter but also contains artefacts due to the electromagnetic field scattered by the resonant nanoparticle. It is demonstrated in~\cite{raab17} that coupling between the plasmonic particle and the fluorescent molecule leads to an incorrect molecular localisation of up to $30$ nm, namely a single-molecule mirage.  We refer to the review~\cite{Willets:2017aa} and references therein for more details.  

The only current solution to this problem is to perform a least-square optimisation on the position of the emitter, simulating the measurements at each step by solving the full Maxwell system with a different source position at every step.  This solution could give satisfying results for the localisation of one emitter, but the computational cost makes this method impossible to use to get a real super-resolution image, as the single emitter localisation procedure needs to be done tens of thousands times to get an image.  

The challenge is therefore to find a model that is accurate enough to precisely describe the interaction between a single emitter and a resonant structure, yet computationally light enough to be able to retrieve the positions of tens of thousands of emitters.

\subsection{Mathematical challenges and recent advances}\label{subsec:challenges}

A natural approach to tackle this problem is to model the response of the plasmonic system with excitation independant quantities. A natural way to do this is to use modal analysis. The idea is to express the measured field as a sum of the field generated by the single point-wise emitter and the contribution from the nanoparticle described by a few modes.
The obvious strength of this description is that the modes can be pre-computed, and the procedure of solving the forward problem can be reduced to computing the modal excitation coefficients.

The main challenge is that the mirage phenomenon happens in a regime where the extinction cross section of the nanoparticle is dominated by the scattering cross section. This happens for larger nanoparticles, of typical size larger than $50$ nm (see \cite[section 4]{ammari2017mathematicalscalar} for more details). 

While there has been plenty of work on the mathematical description of plasmonic resonances as an eigenvalue problem~\cite{grieser2014plasmonic,ando2016analysis,ammari2016surface}, most of these works rely on an electro-static approximation that does not describe well the behaviour of metallic nanoparticles when their size becomes \emph{small but comparable to the wavelength} of the electromagnetic field.  For instance, when the size of the nanoparticle is larger than one tenth of the wavelength (around $50$ nm in practice), the quasi-static theory breaks down~\cite{kelly2003optical,moroz2009depolarization}. It is only recently, via the introduction of a perturbative spectral method, that the field scattered by larger nanoparticles has been precisely described \cite{millien18,baldassari2021modal,AMMARI2022676}. 
In this paper, using these recent perturbed modal decompositions, we tackle the problem of the super-localisation of a single emitter near a plasmonic structure in the regime that is relevant for applications.

\subsection{Contribution of the paper}

Using a simple scalar \emph{toy model} that still encompasses all the mathematical complexity of plasmonic resonances, we show that it is possible to recover the position of a single point-like emitter near a resonant structure up to a precision that is two orders of magnitudes below the diffraction limit with a low computational complexity.

The paper is structured as followed: in section \ref{sec:model} we present the mathematical modelling of the physical system and justify its relevance. Then, in section~\ref{sec:modal}, we give the modal approximation of the field. In section~\ref{sec:localisation}, we move on to introduce the imaging functionals that allow the reconstruction of the dipole's position and orientation. Finally, we show numerical experiments that corroborate our results in section~\ref{sec:num} and give concluding remarks in section~\ref{sec:con}.

\section{Model}
\label{sec:model}

\subsection{Notations}

We consider a system composed of a single point-like dipole source (modelling a fluorescent molecule) and a plasmonic particle embedded in a homogeneous medium. The plasmonic particle occupies a bounded simply connected domain $ D \subset \R^2$ of class $C^{1,\alpha}$ for some $0 < \alpha < 1$. It has characteristic size~$\delta$.  The point-like source  is modelled by an ideal electric dipole of centre $z^*\in\R^2\setminus{\overline{D}}$ and direction $\mathbf{p^*}\in\mathbb{S}^1$. We denote the permittivity of the plasmonic particle by $\varepsilon_D$. The permittivity of the background medium is denoted by $\varepsilon_m$. In other words, the permittivity distribution $\varepsilon$ is given by
\begin{equation*}
\varepsilon=\varepsilon_D\chi \left(D\right)+\varepsilon_m\chi \left(\R^2 \setminus \overline{D}\right).
\end{equation*}
Note that $\varepsilon_D$ depends on the frequency $\omega$.  The model used for $\varepsilon_D$ depends on the type of metal.  It is nonrestrictive to assume that $\varepsilon_D$ is described by a Drude model \cite{sarid2010modern}:
\begin{align*}
\varepsilon_D(\omega) =\varepsilon_0\left( 1-\frac{\omega_p^2}{\omega\left(\omega +i\tau^{-1}\right)}\right),
\end{align*}
where $\omega_p$ is the plasma frequency of the bulk material and $\tau$ is the bulk electron relaxation time.  Any similar model can be used \cite{ordal1983optical}.
 The permeability $\mu$ is constant everywhere, $\mu=\mu_m$. We define the wavenumbers $k_D=\omega\sqrt{\varepsilon_D\mu_m}$ and $k_m=\omega\sqrt{\varepsilon_m\mu_m}$. We denote by $z_D$ the centre of the particle $D$. Up to a translation it is nonrestrictive to assume that $z_D=0$.

Consider $\Omega\supset D $ a domain of characteristic size $R\gg  {k_m}^{-1}$ and $d(D,\partial \Omega)\gg {k_m}^{-1}$. Our goal is to reconstruct $\mathbf{p^*}$ and $z^*$ from the knowledge of far-field data, \emph{i.e.}, the knowledge on $\partial \Omega$ of $u$   the solution of
\begin{align}\label{eq:u}
\nabla\cdot\left(\frac{1}{\varepsilon(x)} \nabla u\right) + \omega^2 \mu \ u = \mathbf{p^*} \cdot \nabla \delta_{z^*} \qquad \tin \R^2,
\end{align} satisfying the outgoing Sommerfeld radiation condition
	\begin{equation*}
	\left| \frac{\partial u}{\partial |x|}-ik_mu\right|=\mathcal{O}\left(|x|^{-3/2}\right), \qquad \mbox{as } |x|\rightarrow \infty,
	\end{equation*}
	uniformly in $x/|x|$, for $\Re{k_m}>0$.
For simplicity we assume that $\Omega= B_R$ the disk centred at the origin of radius $R$. The setting is sketched in Figure~\ref{fig:setting}.

\begin{figure}[H]
		\centering
			\includegraphics[trim={4cm 6cm 2cm 2cm},clip,width=10.5cm]{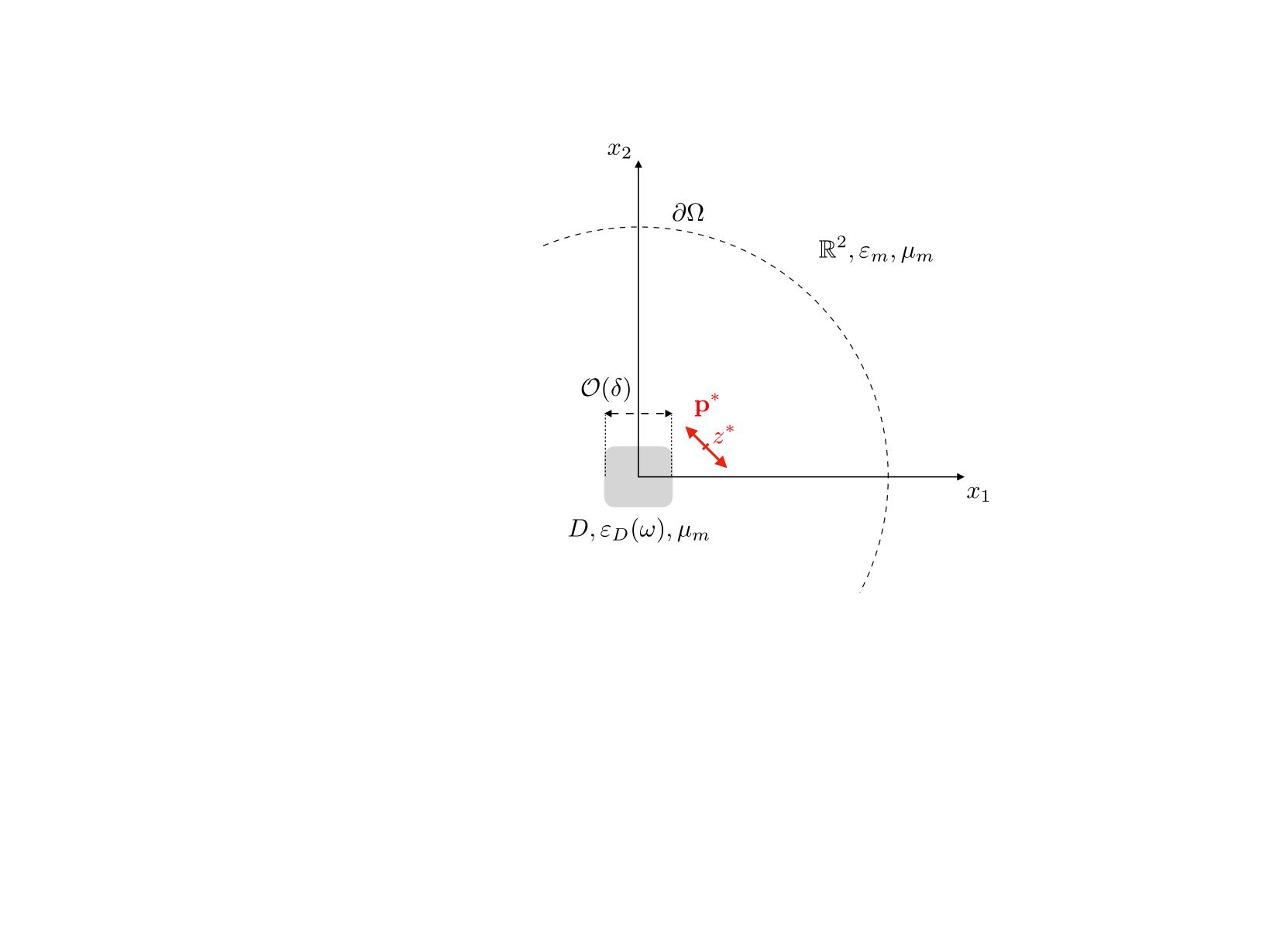} 
		\caption{Setting.}
		\label{fig:setting}
\end{figure}

\begin{proposition} If $\Im \varepsilon_D \neq 0$, then equation \eqref{eq:u} is well posed.
\end{proposition}
\begin{proof}This is a classical result. A recent treatment of this exact case can be found in \cite{ammari2017mathematicalscalar}.

\end{proof}
\begin{remark}The condition $\Im \varepsilon_D \neq 0$ is not sharp,  but it is sufficient. The full treatment of the case $\Im \varepsilon_D \rightarrow 0$ is done in \cite{nguyen2016limiting}.
\end{remark}

\subsection{Relevance of the scalar model}

The scalar Helmholtz equation is the correct way to describe the electromagnetic field propagation in a two-dimensional environment~\cite[Remark 2.1]{moiola2019acoustic}.  It is relevant to model the propagation of an electromagnetic wave in a medium that is invariant along a coordinate. 

Two cases can be modelled. One case is when the electric field is perpendicular to the plane of propagation. It has the form $(0, 0,E_z)$, while the magnetic field has the form $(H_x,H_y,0)$. This case is often called the \emph{transverse electric} case. The other case is when the magnetic field is transverse to the plane of propagation. It is called the \emph{transverse magnetic} case and the magnetic field and electric field are $(0,0,H_z)$ and $(E_x,E_y,0)$, respectively. 

For a non-magnetic medium described by a space-dependent permittivity distribution $\varepsilon(x)$ with $x\in \R^2$ and a constant magnetic permeability,  the transverse magnetic case is the most mathematically interesting case, as the spectral properties of the transmission problem for a small obstacle (of characteristic size small compared to the wavelength) are very similar to the ones of the Maxwell transmission problem. 
Indeed, in the \emph{static regime}, the field $\Ebf$ can be written as $\Ebf := \nabla \tilde{u}:=(E_x,E_y,0):=(\nabla u , 0)$. Then  the Lippman-Schwinger equation associated with problem \eqref{eq:u}, \emph{i.e.}, the equation solved by the electric field inside the particle $\nabla u = (E_x,E_y)$ is \cite{millien18}:
\begin{align}\label{eq:graduinside}
\eta \nabla u (x) - \nabla \int_D\nabla \Gamma^0(x-y)  \cdot \nabla u(y)\dd y = \frac{1}{1-\varepsilon_r} D^2\Gamma^0(x-z^*) \mathbf{p}^*, \qquad x\in D,
\end{align}
where $\Gamma^k$, $k\geq 0$,  is a fundamental outgoing solution to the Helmholtz operator $\Delta +k^2$ given by
\begin{equation}
\Gamma^{k}(x)=
\begin{dcases}
\frac{1}{2\pi}\ln|x| &\mbox{if }k=0, \\
-\frac{i}{4}H_0^{(1)}{(k|x|)} &\mbox{if }k>0,
\end{dcases}
\end{equation} and $\eta=\varepsilon_m/(\varepsilon_m-\varepsilon_D)$ is a contrast parameter depending on the relative permittivity of the obstacle with respect to the  surrounding medium. 

In the three-dimensional case,  the integral equation solved by the electric field $\Ebf$ inside the particle is very similar (see for instance \cite{costabel2012essential} or \cite[equation (3.1)]{millien18})
\begin{align}\label{eq:Einside}
\eta \Ebf (x) - \nabla \int_D\nabla \tilde{\Gamma}^0(x-y)  \cdot \Ebf(y) \dd y = \frac{1}{1-\varepsilon_r}  D^2\tilde{\Gamma}^0(x-z^*) \mathbf{p}^*, \qquad x\in D,
\end{align}
where $\tilde{\Gamma}^k$ is the fundamental solution of the Helmholtz equation in three dimensions.

We can notice a strong similarity between equations \eqref{eq:graduinside} and \eqref{eq:Einside}, and therefore the scalar model is a good starting point.  The full static and low frequency spectral analysis of equation \eqref{eq:Einside} is performed in~\cite{baldassari2021modal}, while the Maxwell case can be found in~\cite{AMMARI2022676}.  
In this paper we only focus on the two-dimensional case.  

\subsection{Limits of the model}

The mathematical results presented here hold when the size of the nanoparticle becomes very small compared to the wavelength or when the dipole source gets very close to the particle. 
Nevertheless, the model used to describe the interaction between the fluorescent molecule and the nanoparticle is incorrect when the different scales in the problem become too small, and fails to properly describe the physics of the problem, especially any quantum effect that may appear. 
We can mention the \emph{quenching} phenomenon that occurs when the molecule gets very close to the nanoparticle (a few nanometers)~\cite{castanie2010fluorescence,delga2014quantum}.
When the nanoparticle becomes very small (of characteristic size less than $10$nm), the size of the fluorescent molecule can become comparable to the size of the nanoparticle itself. Therefore it is not justified to model the fluorescent molecule as a point dipole, and a more sophisticated (quantum) model has to be considered~\cite{neuman2018coupling}.
However, as mentioned in subsection~\ref{subsec:challenges}, the novelty of the method proposed here is that it can be applied to \emph{larger} nanoparticles, of size smaller but comparable to the wavelength of the electromagnetic field emitted by the fluorescent molecule.

\section{Modal Expansion of the solution}
\label{sec:modal}

In the following section, we recall the results from~\cite{baldassari2021modal} that are needed for the localisation procedure.

We will use standard \emph{layer potentials} notations. The definitions and symmetrisation process of the Neumann-Poincar\'e operator on which the theorem relies are recalled in appendix~\ref{appendix:layer}.  We also refer to the book~\cite{AKL} and references therein. 

\begin{definition}
We define the \emph{modes} of the system by
\begin{equation*}
e^\omega_n(x):=\mathcal{S}_{D}^{k_m}\left[\varphi_n\right](x), \qquad x \in \R^2\setminus \overline{D}, ~n\in \N^*.
\end{equation*}
$\mathcal{S}_{D}^{k_m}$ is the single-layer potential and $\varphi_n$ are the eigenvectors of the Neumann-Poincar\'e operator $\mathcal{K}_D^*$.
\end{definition}

\begin{theorem}\label{theo:expansion} The unique solution $u$ of \eqref{eq:u} can be written as:
\begin{align*}
u(x)= \nabla \Gamma^{k_m}(x,z^*) \cdot \mathbf{p^*}+ \sum_{n=1}^N \alpha_n e^\omega_n(x) + \mathcal{E}_N(x), \qquad x \in \R^2\setminus \overline{D}, N\in \N^*,
\end{align*}
where $\alpha_n$ are coupling coefficients between the dipole source and the modes that depend on $z^*$, $\mathbf{p}^*$ and~$\omega$.  $\mathcal{E}_N(x)$ is an error term that depends on the position of the source, the shape of the particle $D$, the number of modes considered, and the observation position $x$. 
The coupling coefficients $\alpha_n$ have a (rather lengthy) explicit formula that is given in the proof.
\end{theorem}

\begin{proof}The theorem was proved in \cite[proposition 5]{baldassari2021modal}.  We recall here for the sake of completeness the expression of the coupling coefficients.
 Define $F$ as \begin{equation*}
F=-\frac{1}{\varepsilon_m}\nubf(x)^\top \mathrm{D}^2 \Gamma^{k_m}(x,z^*)\mathbf{p^*}-\frac{1}{\varepsilon_D}\left(\frac{1}{2}I-\mathcal{K}^{k_D,*}_{D}\right)\left(\mathcal{S}^{k_D}_{D}\right)^{-1}[\nabla\Gamma^{k_m}(x,z^*)\cdot\mathbf{p^*}], \quad x\in \partial D,
\end{equation*} where $\nubf(x)$ is the normal vector on $\partial D$.
We denote by $\lambda_n$ the eigenvalues of the Neumann-Poincar\'e operator on $D$ (see appendix \ref{appendix:layer}) and introduce the contrast parameter:
\begin{equation*}
\lambda(\omega):=\dfrac{\varepsilon_m+\varepsilon_D(\omega)}{2\left(\varepsilon_m-\varepsilon_D(\omega)\right)},
\end{equation*}
as well as
\begin{equation*}\tau_n(\omega)=\left(\frac{1}{\varepsilon_D}-\frac{1}{\varepsilon_m}\right)\left(\lambda(\omega)-\lambda_n\right)+\left(\omega \delta c^{-1}\right)^2 \log{\left(\omega \delta c^{-1}\right)}\tau_{n,1},
\end{equation*}
where $\tau_{n,1}$ is a perturbative spectral parameter defined in definition \ref{def:tau1}.  
Then \cite[proposition 5]{baldassari2021modal} gives us the following expansion for $u$:
\begin{align*}
u(x)=  \nabla \Gamma^{k_m}(x,z^*) \cdot \mathbf{p^*}+ \sum_{n=1}^N \frac{\left\langle F,\varphi_n\right\rangle_{\mathcal{H}^*}}{\tau_n(\omega)}\mathcal{S}_{D}^{k_m}\left[\varphi_n\right](x) + \mathcal{E}_N(x), \qquad x \in \R^2\setminus \overline{D}, N\in \N^*.
\end{align*}
Setting $\alpha_n:= \left\langle F,\varphi_n\right\rangle_{\mathcal{H}^*}/(\tau_n(\omega))$ we get the result. 
\end{proof}
\begin{remark}The coupling coefficients
\begin{align*}
\alpha_n=\frac{\left\langle F,\varphi_n\right\rangle_{\mathcal{H}^*}}{\tau_n(\omega)}
\end{align*} can become very large if $\tau_n(\omega) \ll 1$.  This will happen if $\omega$ is close to a plasmonic resonance frequency.  It is then easy to see that in the case that $\left\langle F,\varphi_n\right\rangle_{H^*}\not=0$ and $\tau_n(\omega) \ll 1$,  the field scattered by the nanoparticle can become very large and contribute to the far field as much as the radiating dipole, making it very difficult to localise the dipole, as the imaging system will be \emph{blinded} by the nanoparticle. 
\end{remark}

\begin{remark}The number of modes that have to be considered to obtain a good approximation of the field is a complicated question because the modal approximation is not uniform, neither with respect to the observation point $x$, nor with the frequency $\omega$ or the parameters of the source ($\mathbf{p}^*, z^*$).  However, in practice, only a few modes are needed to correctly describe the scattered field. This is due to the fact that:
\begin{enumerate}
\item the \emph{coupling coefficient} of the mode $n$ to the source
\begin{align*}
\alpha_n=\frac{\left\langle F,\varphi_n\right\rangle_{\mathcal{H}^*}}{\tau_n(\omega)}
\end{align*} have superpolynomial decay with respect to $n$ for a fixed source and frequency \cite[proposition 3]{baldassari2021modal},
\item high order modes do not contribute to the far field.  It has been shown in~\cite{ando2021surface} that, in the static case ($\omega=0$), the contribution of the modes outside the particle decay exponentially with the order of the mode.  It has not been demonstrated in the dynamic case $\omega \neq 0$ but there are numerical indications that the result probably holds for $\omega \neq 0$.
\end{enumerate}
So even though it is very difficult to get a quantitative estimate on the error term in theorem~\ref{theo:expansion},  in practical situations it is a good approximation. 
\end{remark}

\section{Localisation}
\label{sec:localisation}
We now introduce a method that recovers a dipole source position and orientation in the presence of a resonant nanoparticle at a low computational cost. It is based on the modification of a back-propagation/holography procedure, but the principle of the method works with any linear imaging functional. 

In the absence of a resonant structure, super-localisation is possible through the modelling of a \emph{point spread function} (PSF) of the imaging system, that maps a single point-like source to an image.
Then the location of the point source is obtained by finding the centre of the PSF, which can be done with a precision that depends only on the signal-to-noise ratio. In general, it is achieved with a least-square procedure, i.e. by fitting the centre of the PSF on the image. 

In the presence of the resonant structure, the image will not be the PSF, but the superposition of the PSF at the location of the dipole with  the image of the nanoparticle. 
The current state-of-the-art method to recover the correct position of an emitter near a nano-particle requires solving the forward problem (the full-wave Maxwell equations) for each position of the source \cite{fu2017super}.  This is very costly in terms of computation, making it impossible to apply in practical situations where thousands of isolated single emitters have to be localised quickly. 

The idea behind our method is that even though the field scattered by the nanoparticle cannot be known in advance (as it depends on both the position and the orientation of the dipole),  it can be expressed as a linear combination of  pre-computable quantities (the modes).  This suggests that the distortion of the PSF due to the presence of the nanoparticle is simply a linear combination of pre-computable artefacts.
Therefore, instead of performing the least-square procedure on the position of the centre of the PSF, we optimise on both the centre of the PSF and the coefficients of the linear combination of artefacts.

\subsection{Back-propagation in the absence of a resonant structure}
Super-localisation of the fluorescent molecule is obtained by back-propagating $u(x)$, for $|x|$ large enough, and maximising the back-propagation function.
The precision of the localisation is limited by the signal-to-noise ratio (SNR).

We introduce the following vector-valued imaging functional.
\begin{definition}\label{def:I}
\begin{align*}
\mathcal{I}(z)=\int_{\partial B_R} \overline{\nabla_z \Gamma^{k_m}(x,z)} u(x) \dd \sigma(x), \quad z\in \Omega,
\end{align*}
where the imaging domain $\Omega$ is such that 
\begin{align*}
D\Subset \Omega \Subset B(0,R).
\end{align*}

\end{definition}
We then include a lemma from \cite{ammari2014backpropagation} with its proof for sake of completeness.
\begin{lemma}\label{lem:HK}
Let \begin{align*}
\mathcal{R}(y,z) = \int_{\partial B_R} \overline{\nabla_y \Gamma^{k_m}(x,y)}\nabla_z \Gamma^{k_m}(x,z)^\top \dd \sigma(x), \qquad (y,z)\in \R^2.
\end{align*}
Then \begin{align*}
\lim_{R\rightarrow\infty} \mathcal{R}(z,z) = \frac{k_m}{8} I_2,
\end{align*} where $I_2$ is the identity matrix.
\end{lemma}
\begin{proof}
We start from the Helmholtz-Kirchhoff theorem \cite{ammari2013mathematical} \begin{align*}
\lim_{R\rightarrow\infty} \mathcal{R}(y,z) = -\frac{1}{k_m} \nabla_y \nabla_z \Im\, \Gamma^{k_m} (y,z),
\end{align*} where $\Im$ is the imaginary part.
We can then compute an approximation of $\mathcal{R}$:
\begin{multline*}
\frac{1}{k_m} \nabla_y \nabla_z \Im\, \Gamma^{k_m} (y,z) = \frac{1}{4}\Bigg[ k_m J_0\left(k_m \vert y-z\vert\right) \frac{\left(y-z\right)\left(y-z\right)^\top}{\vert y-z\vert^2}\\- 2\frac{J_1\left(k_m \vert y-z\vert\right)}{\vert y-z\vert} \frac{\left(y-z\right)\left(y-z\right)^\top}{\vert y-z\vert^2} +\frac{J_1\left(k_m \vert y-z\vert\right)}{\vert y-z\vert} I_2  \Bigg].
\end{multline*}
One can see that $\mathcal{R}(y,z)$ decays as $\vert y-z\vert^{-1/2}$ and the imaging functional has a peak at $y=z$. Evaluating at $y=z$ gives the result.
\end{proof}

It is easy to see from lemma \ref{lem:HK} that in the absence of the resonant structure, $\mathcal{I}(z)$ would have an amplitude peak at $z=z^*$ the position of the dipole source and its orientation would be parallel to $\mathbf{p^*}$.
To get $z^*$ and $\mathbf{p^*}$ one needs to solve the following optimisation problem in $\R^4$:
\begin{align}
(z^*,\mathbf{p^*}) = \mathrm{argmin}_{z,\pbf} \left\Vert \mathcal{I}(\cdot) - \mathcal{R}(\cdot, z)\pbf\right\Vert_{L^2(\Omega)}.
\label{eq:minimisationnocorrection}
\end{align}

\subsection{Back-propagation in the presence of the resonant structure}\label{sec:numerics}

\begin{proposition}
In the presence of the nanoparticle, the imaging functional has the form:
\begin{align*}
\mathcal{I}(z)= \mathcal{R}(z,z^*)\mathbf{p^*} + \sum_{n=1}^N \alpha_n I_{e_n}(z) + I_{\mathcal{E}_N}(z), 
\end{align*}
where $\alpha_n$ are the (unknown) coupling coefficients defined in theorem~\ref{theo:expansion},  $I_{e_n}$ the (known) image by the imaging functional of the mode $n$ and $ I_{\mathcal{E}_n}(z)$ the (unknown) image of the error term.
\end{proposition}

\begin{proof}
The result is a direct consequence of theorem \ref{theo:expansion} and the linearity of the imaging functional. Let
\begin{align*}
I_{e_n}(z)= \int_{\partial B_R} \overline{\nabla_z \Gamma^{k_m}(x,z)} e_n(x) \dd \sigma(x),
\end{align*} and
\begin{align*}
I_{\mathcal{E}_N}(z)= \int_{\partial B_R} \overline{\nabla_z \Gamma^{k_m}(x,z)} \mathcal{E}_N(x) \dd \sigma(x).
\end{align*} We obtain the result by applying lemma \ref{lem:HK}.
\end{proof}

To recover $z^*$ and $\mathbf{p^*}$ in the presence of the resonant structure,  we now solve an optimisation problem in $\R^{4+N}$:
\begin{align}\label{eq:minwithcorrection}
\left(z^*,\mathbf{p^*},(\alpha_n^*)_{n=1}^N \right) = \mathrm{argmin}_{z,\pbf,(\alpha_n)_{n=1}^N} \left\Vert \mathcal{I}(\cdot) - \mathcal{R}(\cdot, z)\pbf-\sum_{n=1}^N\alpha_n I_{e_n}(z) \right\Vert^2_{L^2(\widetilde{\Omega})}.
\end{align}

\section{Numerical simulations}
\label{sec:num}
In this section we illustrate the method's efficiency with numerical simulations.
We perform a virtual experiment where a dipole source is placed near a metallic nanoparticle whose permittivity is described by a Drude model.  We consider nanoparticles with various shapes and subject to noise to illustrate the robustness of the method.

\subsection{Geometry and physical parameters}
Throughout this section, we consider the three domains sketched on Figure~\ref{fig:domains} to illustrate our results. The five-petal flower~(a) is defined by $\varrho=\delta(2+0.6\cos(5\theta))$ in polar coordinates. The rounded diamond~(b) is defined by the parametric curve $\zeta(\theta)=2\delta\left(e^{i\theta}+0.066e^{-3i\theta}\right)$, for $\theta\in[0,2\pi]$. The narrow ellipse~(c) semi-axes are on the $x_1$- and $x_2$- axes and are of length $a=1 \delta$  and $b=5\delta$, respectively. 


		\begin{figure}[H]
		\centering
			\includegraphics[trim={1.5cm 4.5cm 0.4cm 4cm},clip,width=15cm]{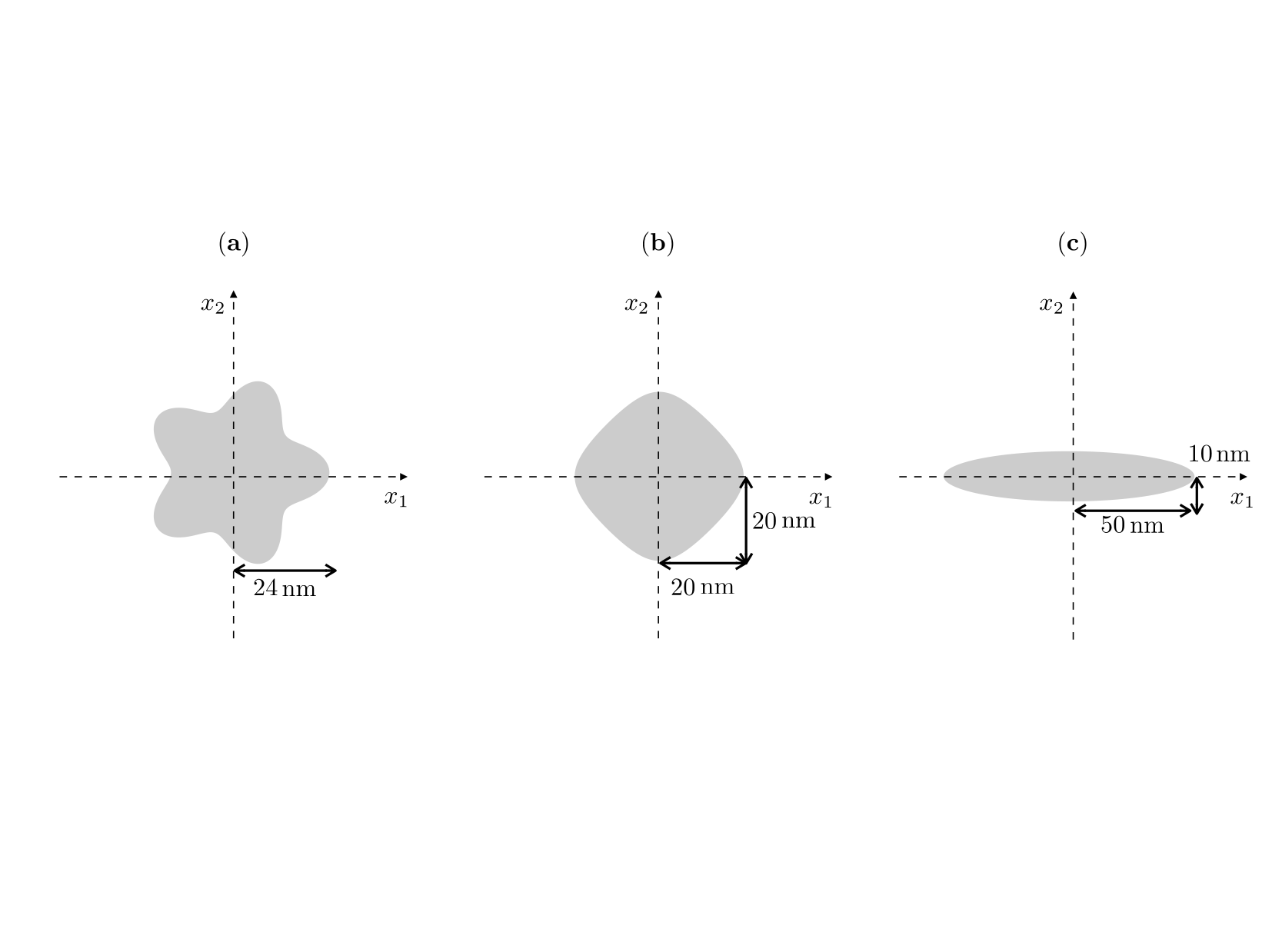} 
		\caption{Sketch of the three reference domains: the five-petal flower (a), the rounded diamond (b) and the narrow ellipse (c).}
		\label{fig:domains}
	\end{figure}	
We recall the permittivity inside the particle is described by a Drude model
\begin{align*}
\varepsilon_D(\omega) =\varepsilon_0\left( 1-\frac{\omega_p^2}{\omega\left(\omega +i\tau^{-1}\right)}\right).
\end{align*}
The physical parameters are chosen as $\omega_p=2\cdot 10^{15}$ Hz, $T=10^{-14}$s, $\varepsilon_0=8.854187128 \cdot 10^{-12}$ Fm$^{-1}$, $\mu_0=4\pi \cdot 10^{-7}$ Hm$^{-1}$ and $\delta=10^{-8}$ m. 
%
%
\subsection{Modes and resonances}
For each shape, we compute a  list of $N$ plasmonic resonant frequencies $(\omega_n)_{n\in [1,N]}$	and the associated modes. 
We give an illustration for the rounded diamond.
\begin{figure}[H]
\centering
			\includegraphics[,clip,width=\linewidth]{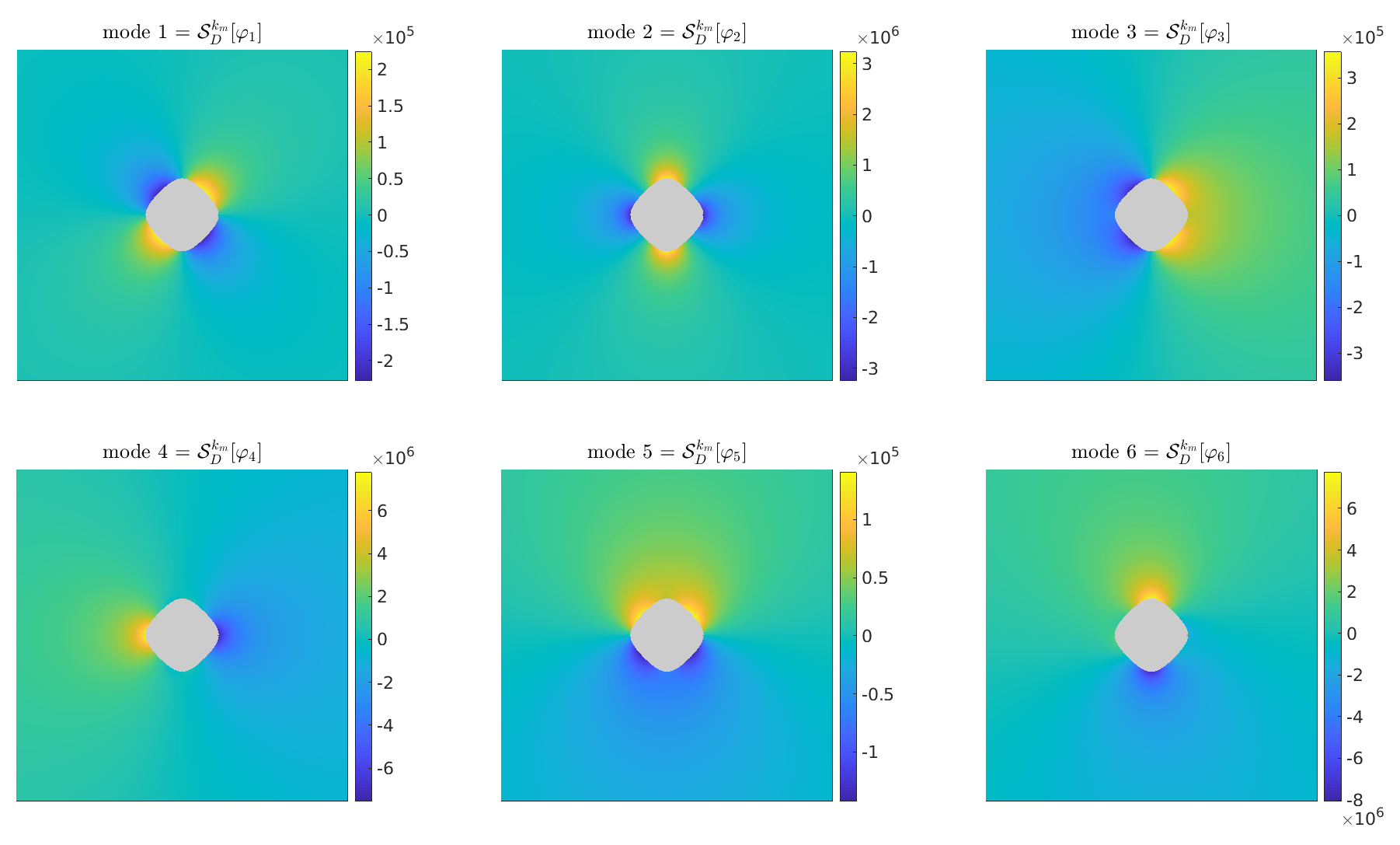} 		
			\caption{First $6$ modes for the rounded diamond,  corresponding to $\omega=1.5050 \cdot 10^{15}$ Hz.}
\end{figure}

\subsection{Imaging experiment}

For each shape we pick a resonant frequency $\omega_n$ and place an oscillating dipole at frequency $\omega_n$ near the nanoparticle.  We choose the frequency that corresponds to a mode that has a dipole radiation pattern, to maximise the effect of the nanoparticle's presence on the far-field image and make the localisation of the dipole source as hard as possible. 

The scattered field (solution of \eqref{eq:u}) is then computed with a system of boundary integral equations (that uses layer potential techniques) detailed in appendix \ref{appendix:layer}.
The boundary of the nanoparticle is discretised with $2^8$ points.

The  field (amplitude and phase) is then collected in the far field on a disk of radius 3000$\delta$.
The measured field is then back-propagated using the imaging functional introduced in definition \ref{def:I}.

\subsection{Localisation by optimisation and mirage}
The localisation (without correction) of the particle is done by fitting a focal spot on the imaging functional, as described in equation~\eqref{eq:minimisationnocorrection}.
The optimisation is done using the \emph{fmincon} MATLAB\textsuperscript{\textregistered} function.  We constrain the search to a $300\times 300 \text{nm}^2$ square. The position of the initial guess does not influence the result. 

For the corrected functional, we use the \emph{fmincon} function to search for a solution of equation~\eqref{eq:minwithcorrection}. The number of modes to use for the correction will be discussed later in this section. 

We show, on Figure~\ref{fig:mirage_flo_dia}, examples of the results of the optimisation procedure for the three different domains.
		\begin{figure}[H]
		\centering
			\includegraphics[trim={9cm 17cm 6cm 3cm},clip,width=16cm]{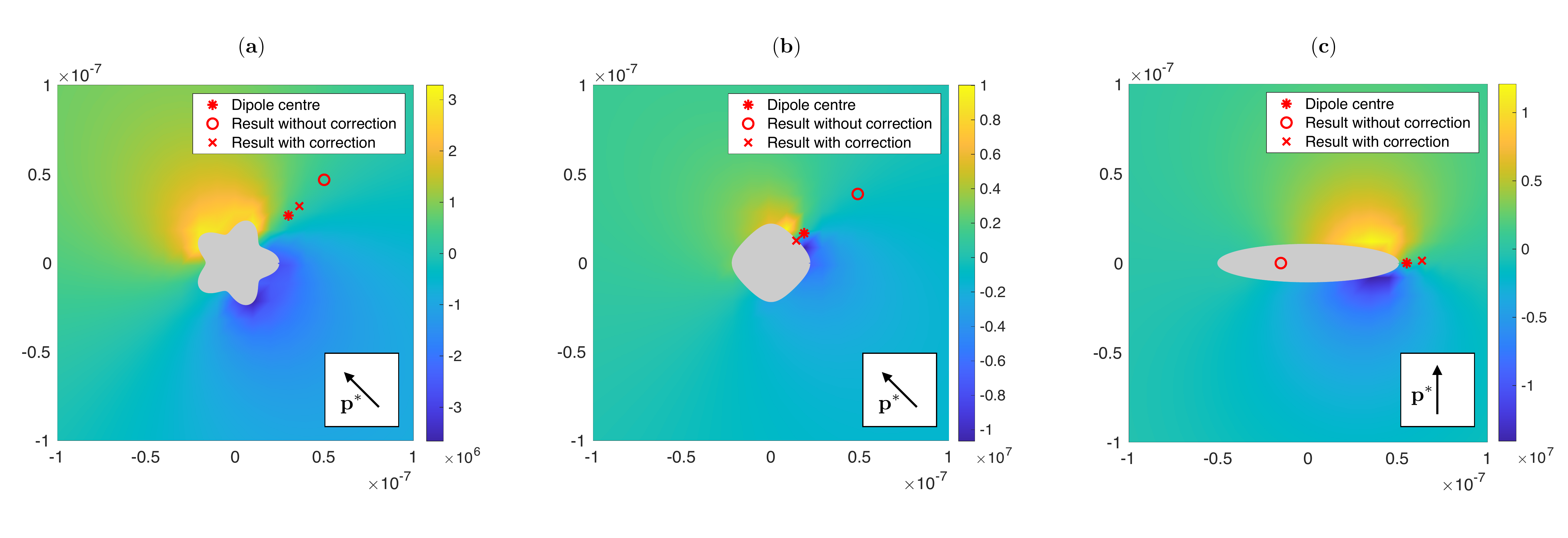} 
		\caption{Mirage and optimisation on the three domains: for the flower (a), the localisation error is $29$nm before correction and $8$nm. For the diamond (b), the error is $37$nm before correction and $5.6$nm. The dipole direction is $(-1,1)$. For the ellipse (c), the localisation error is $71$nm before correction and $8$nm after and the dipole direction is $(0,1)$.}
		\label{fig:mirage_flo_dia}
	\end{figure}

\subsection{Sensitivity analysis}

On Figure~\ref{fig:err} we plot the error on the localisation (a) and orientation (b) of the dipole with respect to the distance between the emitter and the nanoparticle. We initially place the dipole at $z^*= [18.65,16.65]$nm and fix its orientation ($\mathbf{p}^*=(-1,1)$). We gradually increase the distance to the nanoparticle by translating the dipole along a radial component of a disk centred at the origin.
As expected, the uncorrected functional does not behave well when the source is close to the nanoparticle. Note that the error is still one order of magnitude below the usual diffraction limit.  The localisation error on the corrected functional is one order of magnitude lower than the uncorrected one. The correction is a massive improvement if the emitter is close enough to excite the nanoparticle. When the emitter is far away and the coupling is weaker, both procedures have similar performances. 
\begin{figure}[H]
	\centering
	\includegraphics[trim={1cm 0cm 1.5cm 0cm},clip,width=7.5cm]{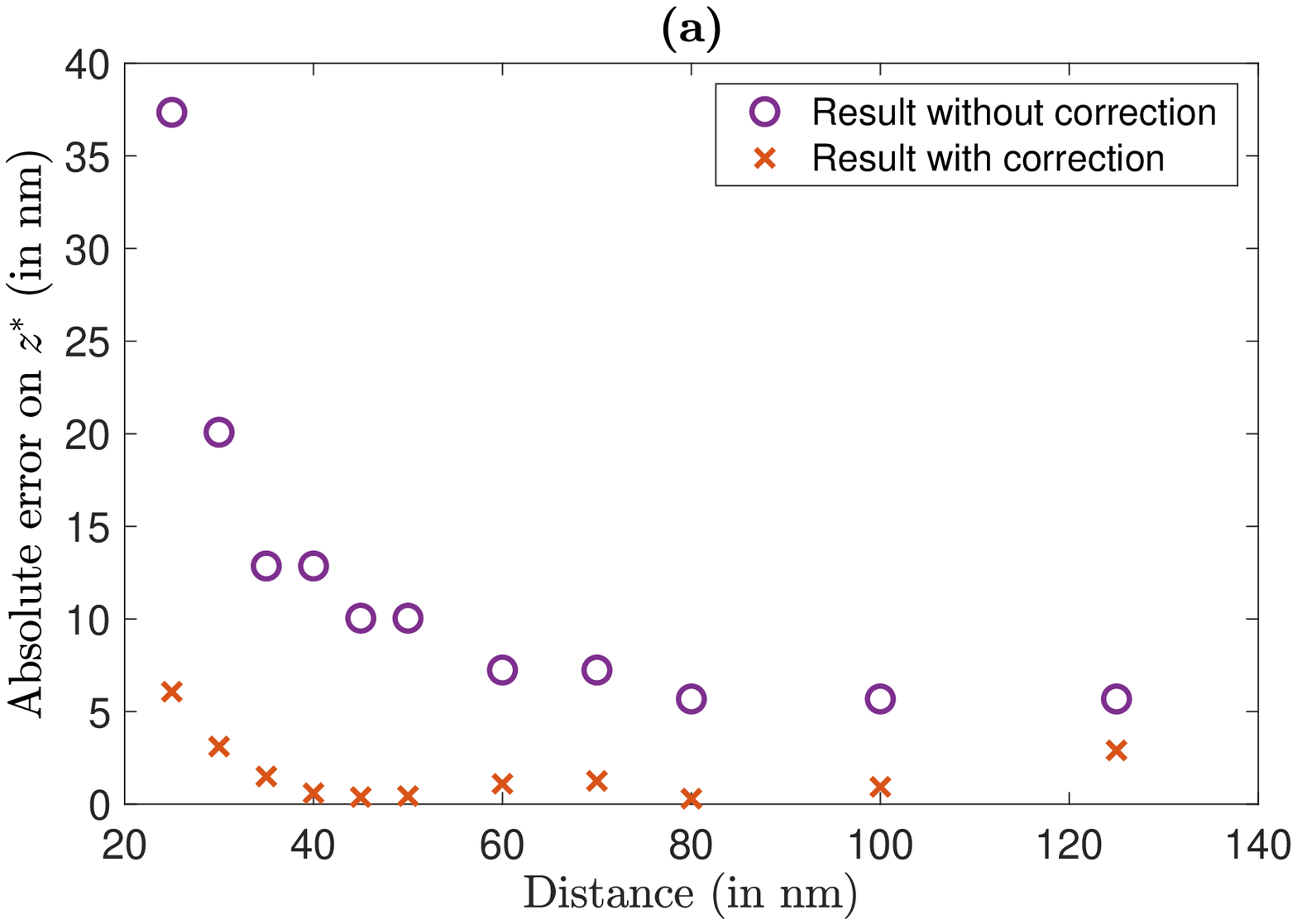} 
		\includegraphics[trim={1cm 0cm 1.5cm 0cm},clip,width=7.5cm]{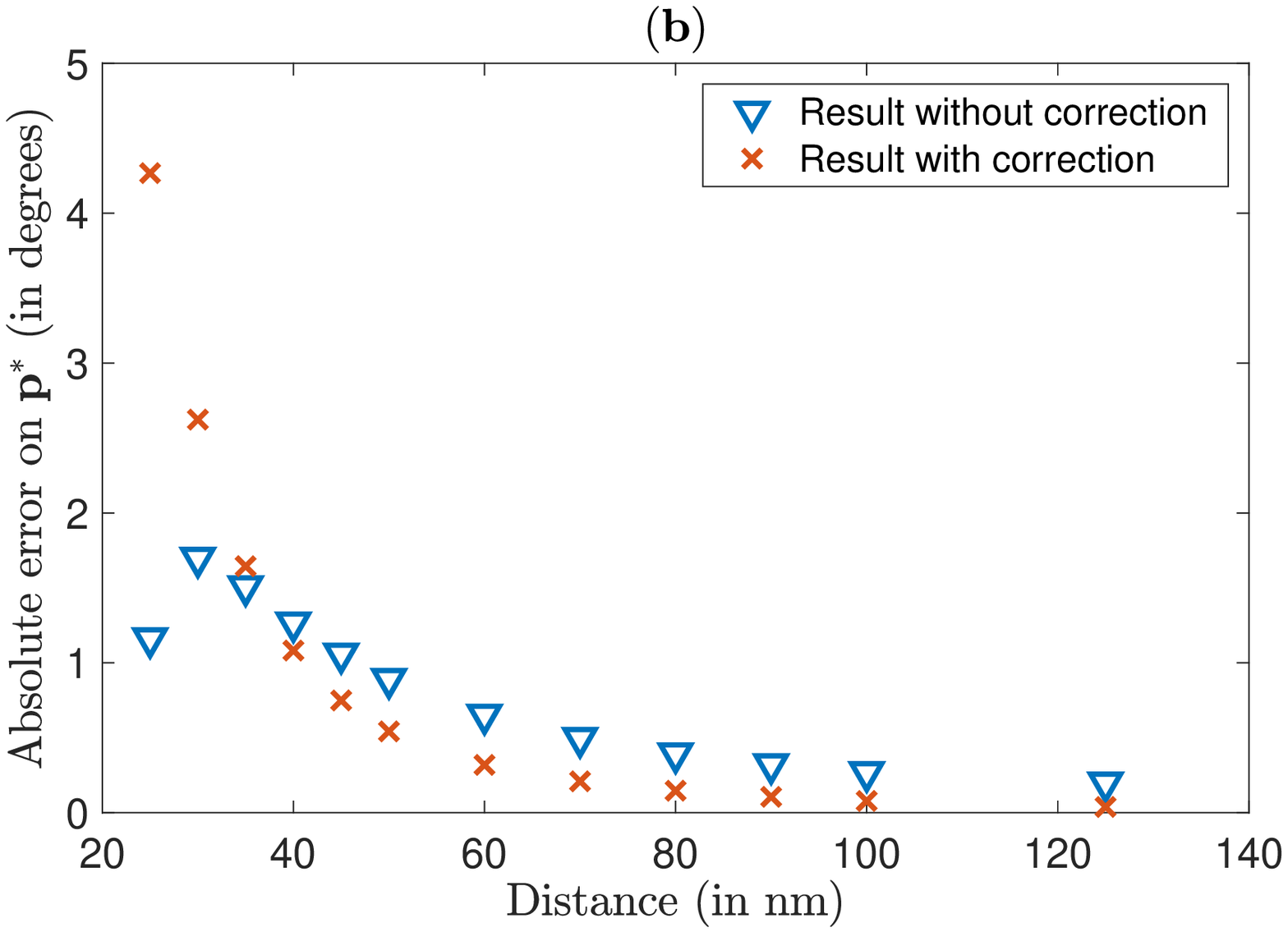} 
	\caption{(a) Absolute error on the position $z^*$ in nanometers with (orange crosses) and without (purple circles) correction against the distance between the emitter and the rounded diamond boundary. (b) Absolute error on the orientation $\mathbf{p}^*$ in degrees with (orange crosses) and without (blue triangles) correction against the distance between the emitter and the rounded diamond boundary. The initial position is $z^*= [18.65,16.65]$nm. The dipole direction is $(-1,1)$.}
	\label{fig:err}
\end{figure}

\subsection{Stability with respect to measurement noise}
We compute the imaging functional with a set of data perturbed by Gaussian white noise. We consider the case in which the measured field $u(x)$ is corrupted by an additive noise $\nu(x)$, $x \in \partial B_R$.
We assume that $\partial B_R$ is covered uniformly with sensors and that the additive noises $\nu(x)$ have independent and identically distributed $\mathcal{N}(0,\sigma_\text{noise})$ entries. Hence, the entries of $\nu(x)$ are
independent Gaussian random variables with mean zero and variance 
$$
\sigma_\text{noise}= \sigma_0 \| u \|_F/\sqrt{N},  $$
where $||~||_F$ is the Frobenius norm, $\sigma_0$ the noise percentage and $N$ the number of sensors.

For each noise level, we average the results over $100$ realisations. Figure~\ref{fig:mirage_flo_noise} presents the results of computational experiments in the case in which the nanoparticle is the diamond. It shows that the corrected imaging functional performs well at high levels of noise, which means that the corrected imaging functional is robust with respect to additive measurement noise. 
	
			\begin{figure}[H]
		\centering
			\includegraphics[trim={0.7cm 0cm 1.5cm 0cm},clip,width=7.7cm]{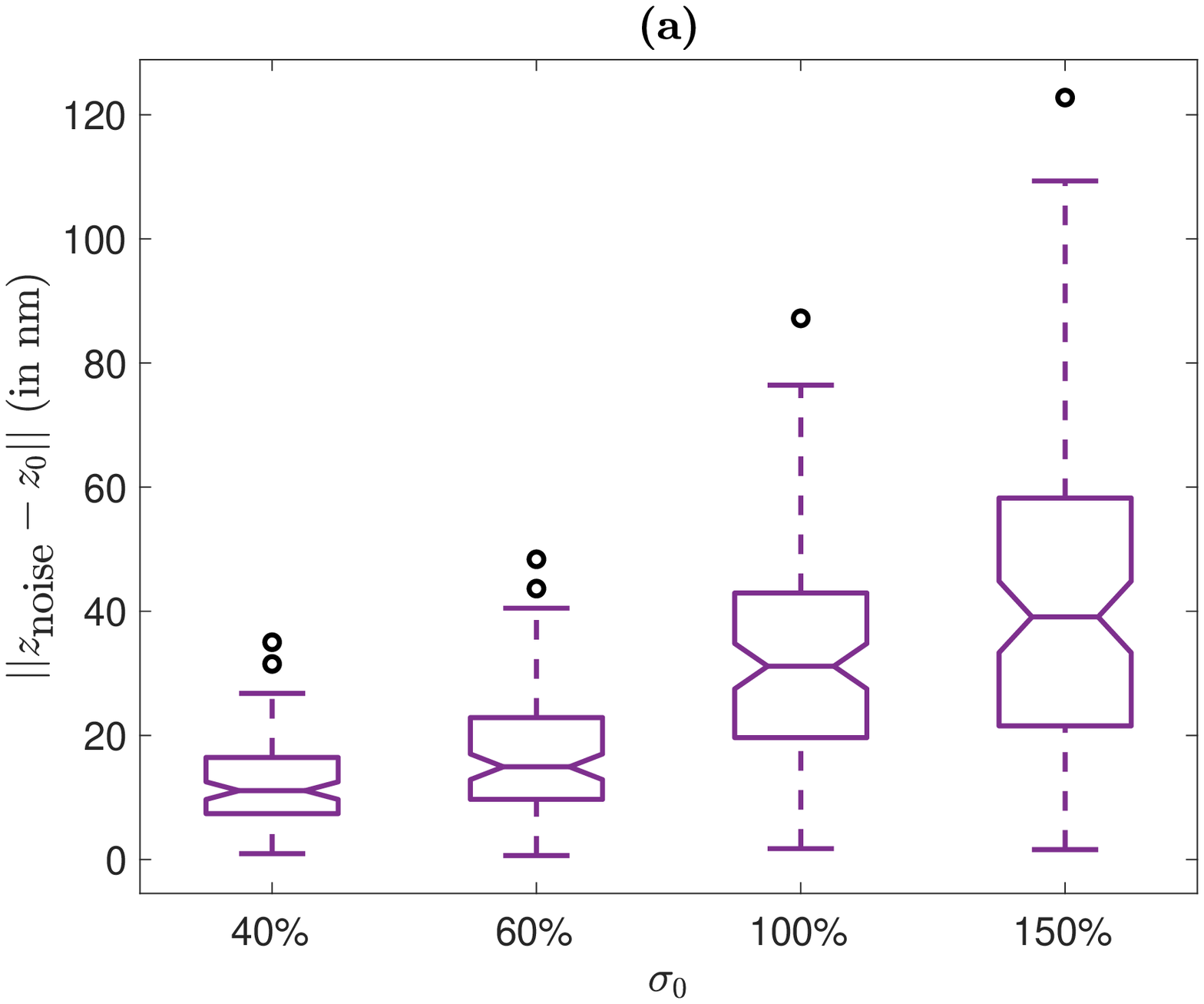} 
			\includegraphics[trim={.7cm 0cm 1.5cm 0cm},clip,width=7.7cm]{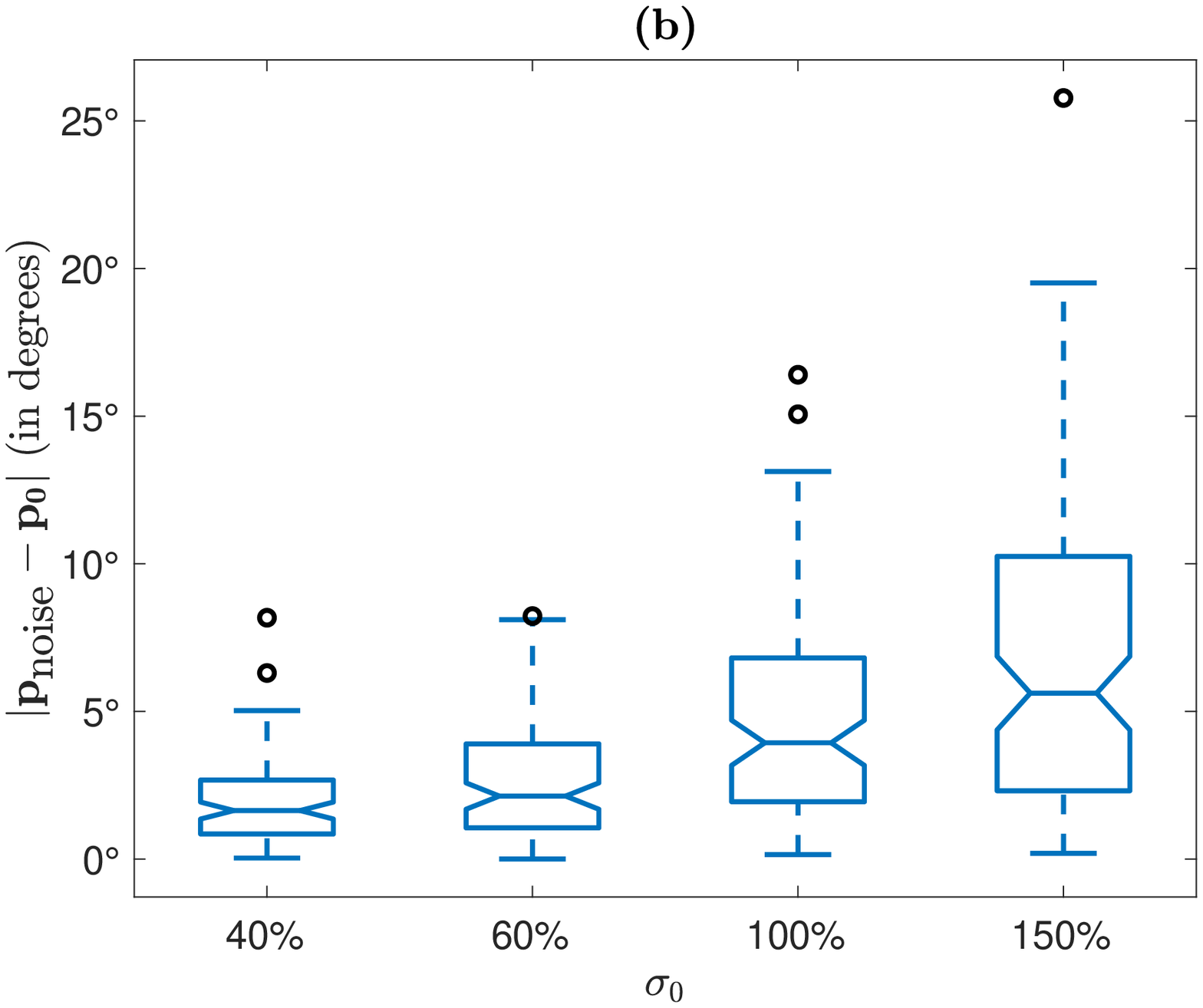} 
		\caption{Boxplots showing the error on the position (a) and orientation (p) for noise levels $\sigma_0$ varying from $40\%$ to $150\%$. The boxplots show the median, 25- and 75-quartiles, the notches correspond to the 95 confidence interval for 100 realisations. The outliers are shown as black symbols.}
		\label{fig:mirage_flo_noise}
	\end{figure}

\subsection{Number of modes}
For the numerical experiments of this section we have assumed to know in advance the pre-computable artefacts that describe the distortion of the PSF due to the presence of the nanoparticle, \emph{i.e.}, the modes excited by the fluorescent molecule. In practice, this may not the case. However, we know that only a few modes are excited and contribute significantly to approximate the scattered field. One can ask whether optimising over excited \emph{and} unexcited modes will give results as good as the ones obtained by optimising over the excited modes only. 

For these simulations we consider the rounded diamond. The dipole position is $[18.65,16.65]$nm. The initial direction is $(-1,1)$. In this case, we can check numerically that only modes 4 and 6 suffice to approximate the field.  

On Table~\ref{tab:1}, we show the error on the position and orientation for the first $N$ modes. As expected, the error becomes small as soon as $N\geq 6$, \emph{i.e.}, modes 4 and 6 are considered, regardless of the number of unexcited modes used for the optimisation procedure.    
	
\begin{table}[h!]
	\centering
	\begin{tabular}{|c||c|c|c|c|c|c|}
		\hline 
		$N~\text{modes}$ & 2 &  3 & 4 & 5 & 6 &  7  \\
		\hline \hline 
		$||\bf{p}_N-\bf{p}^*||~\text{(in deg)}$  & 0.57 & 89.2  & 13.9 & 45.0 & 0.57 &  0.30  \\
		\hline
		$||z_N-z^*||~\text{(in nm)}$ & 116 &  145 & 116&  166 & 4.74 &  4.22  \\
		\hline
	\end{tabular}
	\caption{Error on the position and orientation for the first $N$ modes, $N=2..7$. }
	\label{tab:1}
\end{table}

\section{Concluding remarks}
\label{sec:con}
In this paper we have shown a new method to recover the localization of a single isolated source in a resonant environment with a low computational cost.  Using the results of \cite{baldassari2021modal,AMMARI2022676} and pre-computing the modes of the resonant structure, we transform the  \emph{inverse source problem for the wave equation} into a simple minimization in $\R^d$ with $d$ being the number of modes to consider, usually less than $10$.  We can recover the position of the source with a precision that is two orders of magnitude below the usual diffraction limit.   Since most of the computational cost is the pre-computation of the modes, the position of many single sources can be recovered very quickly, as is needed in real \emph{PALM} or \emph{STORM} experiments, where the position of thousands of emitters are needed to produce an image.  This paper acts like a proof of principle for this method.  In a forthcoming paper,  we will adapt the numerics for the $3D$ Maxwell case that is more suited for the applications, and consider a more complex structured environment as well as amplitude only measurements.

\pagebreak
\appendix

\section{Layer potentials}\label{appendix:layer}
The results presented in this section are all classical and can be found in the books \cite{AKL,ammari2013mathematical}. 
\subsection{Operator}
For $k\geq 0$, a fundamental solution to the Helmholtz operator $\Delta +k^2$ is given by
\begin{equation*}
\Gamma^{k}(x)=
\begin{dcases}
\frac{1}{2\pi}\ln|x| &\mbox{if }k=0, \\
-\frac{i}{4}H_0^{(1)}{(k|x|)} &\mbox{if }k>0,
\end{dcases}
\end{equation*}
for $x\neq 0$, where $H_0^{(1)}$ is the well-known Hankel function of the first kind and order zero.
Let $\mathcal{S}^k_D$ be the single layer potential, defined by
\begin{equation}\label{eq:singlelayer}
\mathcal{S}^k_D[\phi](x)=\int_{\partial D}\Gamma^k(x,y)\phi(y)\mathrm{d}\sigma(y),\qquad x\in \R^2,
\end{equation}
for $\phi\in L^2(\partial D)$.
We also define the Neumann-Poincar\'e operator by
\begin{equation}\label{eq:k*}
\mathcal{K}^{k,*}_D[\phi](x)=\int_{\partial D}\frac{\partial \Gamma^k(x,y)}{\partial \nu(x)}\phi(y)\mathrm{d}\sigma(y),\qquad x\in \partial D,
\end{equation}
for $\phi\in L^2(\partial D)$, where $\partial /\partial \nu(x)$ denotes the outward normal derivative at $x\in \partial D$.
When $k=0$, we omit the superscript and write $\mathcal{S}_D$ and $\mathcal{K}^{*}_D$ for simplicity. Let $\left\langle \cdot,\cdot\right\rangle_{-1/2,1/2}$ be the duality pairing between $H^{-\frac{1}{2}}(\partial D)$ and $H^{\frac{1}{2}}(\partial D)$, where $H^{\frac{1}{2}}(\partial D)$ is the Sobolev space of order 1/2.
The single layer potential is, in general, not invertible in $L^2(\partial D)$. Let us introduce 
\begin{equation*}
\widetilde{\mathcal{S}}_D[v] = 
			\begin{cases}
			\mathcal{S}_D[v] & \text{if } \langle v, \chi(\partial D)\rangle_{-\frac{1}{2}, \frac{1}{2}} = 0,\\
			-\chi(\partial D) & \text{if } v = \varphi_0,
			\end{cases}
\end{equation*}
with $\varphi_0$ being the unique (in the case of a single particle) eigenfunction of $\mathcal{K}^*_D$ associated with eigenvalue $1/2$ such that $\langle \varphi_0, \chi(\partial D)\rangle_{-\frac{1}{2}, \frac{1}{2}}=1$.
The operator $\mathcal{K}^{*}_D:H^{-1/2}(\partial D)\rightarrow H^{-1/2}(\partial D)$ is compact and the following Plemelj's symmetrisation principle identity (also
known as Calder\'on) holds on $H^{-1/2}(\partial D)$:
\begin{equation*}
\widetilde{\mathcal{S}}_D\mathcal{K}^*_D=\mathcal{K}_D\widetilde{\mathcal{S}}_D.
\end{equation*}
Let $\mathcal{H}^*(\partial D)$ be the space $H^{-1/2}(\partial D)$ equipped with the following inner product:
\begin{equation*}
<u,v>_{\mathcal{H}^*(\partial D)}=-<\widetilde{\mathcal{S}}_D[v],u>_{1/2,-1/2}.
\end{equation*}
The Neumann-Poincar\'e operator $\mathcal{K}^*_D$ is self-adjoint in the Hilbert space $\mathcal{H}^*(\partial D)$. Let $(\lambda_n,\varphi_n)_{n\in\N}$ be the eigenvalue and normalised eigenfunction pair of $\mathcal{K}^*_D$ in $\mathcal{H}^*(\partial D)$. Then $\lambda_0=1/2$, $-1/2<\lambda_n<1/2$ for $n\geq1$ and $\lambda_n\rightarrow 0$ as $n\rightarrow+\infty$.

\subsection{Solution for the wave equation}
Let $H^{1/2}(\partial D)$ be the usual Sobolev space and let $H^{-1/2}(\partial D)$ be its dual space with respect to the duality pairing $\left\langle \cdot,\cdot\right\rangle_{\frac{1}{2},-\frac{1}{2}}$. The field $u$ can be represented using the single-layer potentials $\mathcal{S}^{k_c}_D$ and $\mathcal{S}^{k_m}_D$, introduced in equation \eqref{eq:singlelayer}, as follows:
	\begin{equation} \label{eq:scalar_solution}
	u(x)=\begin{dcases}
	\mathcal{S}^{k_c}_D[\Phi](x), & x\in D,\\
	u^\text{in}(x)+\mathcal{S}^{k_m}_D[\Psi](x), & x\in \mathbb{R}^d \setminus \overline{D},
	\end{dcases}
	\end{equation}
	where the pair $(\Phi, \Psi) \in H^{-\frac{1}{2}}(\partial D)\times H^{-\frac{1}{2}}(\partial D)$ is the unique solution to
	\begin{equation}\label{eq:integraleqsystem}
	\begin{dcases}
	\mathcal{S}^{k_m}_D[\Psi](x)-\mathcal{S}^{k_c}_D[\Phi](x)=F_1,& x\in\partial D,\\
	\frac{1}{\varepsilon_m}\left(\frac{1}{2}I+\mathcal{K}^{k_m,*}_D\right)[\Psi](x)+\frac{1}{\varepsilon_c}\left(\frac{1}{2}I-\mathcal{K}^{k_c,*}_D\right)[\Phi](x)=F_2, & x\in\partial D,
	\end{dcases}
	\end{equation}
	and
	\begin{equation*}
	F_1=-u^\text{in}(x), \qquad F_2=-\frac{1}{\varepsilon_m}\frac{\partial u^\text{in}(x)}{\partial \nu}, \qquad x \in \partial D,
	\end{equation*}
	where $\mathcal{K}^{k_m,*}_D$ is the Neumann-Poincar\'e operator introduced in equation \eqref{eq:k*} and $u^{\text{in}}= \mathbf{p}^* \cdot \nabla \Gamma^k(\cdot,z^*)$.

\subsection{Definition of $\tau_{n,1}$}

	\begin{lemma}
		For $k$ small enough, the two-dimensional boundary operator $\widehat{\mathcal{S}}_D^k:\mathcal{H}^*(\partial D)\rightarrow\mathcal{H}^*(\partial D)$ defined as
		\begin{equation}
		\widehat{\mathcal{S}}_D^{k}[\phi](x)=\mathcal{S}_D^0[\phi](x)+\eta_k\int_{\partial D}\phi(y)\mathrm{d}\sigma(y),
		\end{equation}
		is invertible and
		\begin{equation}
		\left(\widehat{\mathcal{S}}_D^k\right)^{-1}=\widetilde{\mathcal{S}}_D^{-1}-\left\langle\widetilde{\mathcal{S}}_D^{-1}[\cdot],\varphi_0\right\rangle_{\mathcal{H}^*(\partial D) }\varphi_0-\mathcal{U}_k,
		\end{equation}
		where $$\mathcal{U}_k=\dfrac{\left\langle\widetilde{\mathcal{S}}_D^{-1}[\cdot],\varphi_0\right\rangle_{\mathcal{H}^*(\partial D)}}{\mathcal{S}_D[\varphi_0]+\eta_k} \varphi_0$$ and $\eta_k=(1/2\pi)(\log{k}+\gamma-\log{2})-i/4,$
		with the constant $\gamma$ being the Euler constant. Note that $\mathcal{U}_k=\mathcal{O}(1/ \log k)$.
	\end{lemma}

	\begin{definition}\label{def:tau1}
\begin{equation*}
\tau_{n,1}=\left\langle\mathcal{A}_{D,1}\varphi_n,\varphi_n\right\rangle_{\mathcal{H}^*(\partial D)}
\end{equation*}

with
\begin{equation*}
		\mathcal{A}_{D,1} =\frac{1}{\varepsilon_m} \mathcal{K}^{(1)}_{D,1}(I -\mathcal{P}_{\mathcal{H}^*_0}) +  \left(\frac{1}{2}I - \mathcal{K}^*_D\right) \widetilde{\mathcal{S}}_D^{-1} \mathcal{S}_{D,1}^{(1)} \left(\frac{1}{\varepsilon_D}I - \frac{1}{\varepsilon_m } \mathcal{P}_{\mathcal{H}^*_0} \right),
\end{equation*}
	\end{definition}

\bibliographystyle{siam}
\bibliography{references}

\end{document}